\documentclass[12pt,reqno]{amsart}

\usepackage{amsfonts}
\usepackage{amsmath,amssymb,mathtools}
\usepackage{microtype}
\usepackage[margin=1in]{geometry}
\usepackage{xcolor}
\usepackage{graphicx}
\usepackage{subcaption}
\usepackage{hyperref}

\hypersetup{
 colorlinks=true,
 linkcolor=blue!55!black,
 citecolor=blue!55!black,
 urlcolor=blue!55!black,
 pdftitle={Non-asymptotic Bounds for the Average Singular Value of a Complex Gaussian Matrix},
 pdfauthor={Lu\'{\i}s Daniel Abreu and Pratik Patil},
 pdfkeywords={complex Gaussian matrices, average singular value, three-term recurrence, continuous dual Hahn polynomials, Wishart matrices, Laguerre polynomials, dimension monotonicity, asymptotic expansion}
}

\newtheorem{theorem}{Theorem}
\newtheorem{proposition}[theorem]{Proposition}
\newtheorem{lemma}[theorem]{Lemma}
\newtheorem{corollary}[theorem]{Corollary}
\theoremstyle{remark}

\numberwithin{equation}{section}

\DeclareMathOperator{\Real}{Re}
\DeclareMathOperator{\tr}{tr}

\begin{document}

\title[Non-asymptotic bounds for the average singular value]{Non-asymptotic bounds
for the average singular value of a complex Gaussian matrix}

\author{Lu\'{\i}s Daniel Abreu}
\address{Faculty of Mathematics, University of Vienna, Vienna, Austria}
\email{abreuluisdaniel@gmail.com}
\author{Pratik Patil}
\address{Department of Statistics and Data Sciences, University of Texas at
Austin, Austin, USA}
\email{pratikpatil@utexas.edu}

\subjclass[2020]{Primary 60B20; Secondary 15A18, 15B52, 33C45}
\keywords{Complex Gaussian matrices, average singular value, three-term recurrence, continuous dual Hahn polynomials, Wishart matrices, Laguerre polynomials, dimension monotonicity}

\date{\today}

\begin{abstract}
Let $G_{d}$ be a $d \times d$ matrix with independent standard complex Gaussian entries, let $\alpha_{\mathbb{C}}(d)$ be the expected average singular value of $G_{d}/\sqrt{d}$, and set $\Delta_d := \alpha_{\mathbb{C}}(d)-\alpha_{\mathbb{C}}(d+1)$.
The statistic $\alpha_{\mathbb{C}}(d)$ admits a variational representation as an expected normalized maximum over the unitary group and governs approximation guarantees for the little Grothendieck problem over the unitary group and related unitary registration problems.
We obtain a strictly positive lower bound and an upper bound for $\Delta_d$, both valid in every dimension, together with corresponding bounds for $\alpha_{\mathbb{C}}(d)$ around the Marchenko--Pastur limit.
These bounds match the sharp leading behavior of complete asymptotic expansions for both quantities, whose coefficients are explicitly computable.
The proof combines a three-term recurrence for $Y_d = d^{3/2}\alpha_{\mathbb{C}}(d)$, obtained from its continuous dual Hahn representation, along with singularity analysis of the underlying Laguerre moment generating function.
\end{abstract}

\maketitle

\section{Introduction}

\subsection{Average singular value of a complex Gaussian matrix}

For each integer $d \geq 1$, let $G_{d}$ be a $d \times d$ random matrix with independent standard complex Gaussian entries, each having density $\pi^{-1}e^{-|z|^{2}}$ on $\mathbb{C}$ (so its mean is zero and its second absolute moment is one).
While the extremal singular values of these matrices have been studied extensively \cite{EdelmanSIAM}, our focus is the expected average singular value of the normalized matrix $G_{d}/\sqrt{d}$, defined as
\begin{equation}
\alpha_{\mathbb{C}}(d)
:= \mathbb{E}\left[
\frac{1}{d}\sum_{j = 1}^{d}\sigma_{j}\!\left(\frac{G_{d}}{\sqrt{d}}\right)
\right].
\label{eq:def-alpha}
\end{equation}

We first record a useful variational interpretation of this quantity.
Let $\mathcal{U}_{d}$ denote the group of $d \times d$ unitary matrices (that is, those satisfying $UU^{H} = U^{H}U = I_{d}$).
Given a singular value decomposition $A = V\Sigma W^{H}$, let $\mathcal{P}(A) := VW^{H}$ denote the corresponding unitary polar factor.
The case $m = 1$, $n = d$ of Ky Fan's maximum principle for singular values \cite[Theorem 1, formula (4)]{KyFan} implies that, for every matrix $A \in \mathbb{C}^{d\times d}$,
\begin{equation}
\sum_{k = 1}^{d}\sigma_{k}(A)
= \max_{U \in \mathcal{U}_{d}}\Real \tr(U^{H}A)
= \Real \tr\left(\mathcal{P}(A)^{H}A\right).
\label{KF}
\end{equation}
Applying (\ref{KF}) to $G_d$ and using (\ref{eq:def-alpha}), we obtain
\begin{equation}
\alpha_{\mathbb{C}}(d)
= \frac{1}{d\sqrt{d}}\,\mathbb{E}\!\left[
\max_{U \in \mathcal{U}_{d}}\Real \tr(U^{H}G_{d})
\right]
= \frac{1}{d\sqrt{d}}\,\mathbb{E}\!\left[
\Real \tr\left(\mathcal{P}(G_{d})^{H}G_{d}\right)
\right].
\label{eq:alpha-variational}
\end{equation}
Thus $\alpha_{\mathbb{C}}(d)$ is the expected normalized optimal value of a Gaussian linear functional over the unitary group.
Although we have not found this variational interpretation stated explicitly, it is implicit in the polar decomposition rounding procedure used by Bandeira, Kennedy, and Singer \cite[Algorithm 3 and Lemma 6]{bandeira2016approximating}, which replaces a Gaussian matrix by its unitary polar factor.
Related variational problems arise in Procrustes analysis \cite{Schonemann} and, with the orthogonal group in place of the unitary group, in global point cloud registration \cite{ChaudhuryKhooSinger}.

Both the value at $d = 1$ and the limit as $d \to \infty$ can be computed explicitly.
Direct polar integration gives $\alpha_{\mathbb{C}}(1) = \mathbb{E}|G_{1}| = \sqrt{\pi}/2$.
As $d \to \infty$, the Marchenko--Pastur law with parameter one, together with the corresponding convergence of the first moments, gives~\cite{marchenko1967distribution,cacciapuoti2013local,tulino2004random}:
\begin{equation}
\lim_{d \to \infty}\alpha_{\mathbb{C}}(d)
= \int_{0}^{4}\sqrt{x}\,\frac{1}{2\pi}\sqrt{\frac{4-x}{x}}\,dx
= \frac{8}{3\pi}.
\label{eq:alpha-limit}
\end{equation}
These endpoint values do not by themselves determine how $\alpha_{\mathbb{C}}(d)$ varies in finite dimensions.
Bandeira, Kennedy, and Singer conjectured \cite[Definition~2, p.~437]{bandeira2016approximating}\footnote{Their Gaussian matrix has independent complex $\mathcal{N}(0,1/d)$ entries and therefore has the same distribution as $G_{d}/\sqrt{d}$.} that $\alpha_{\mathbb{C}}(d)$ is strictly decreasing in $d$.
In our notation, this is the assertion that
\begin{equation*}
\Delta_{d} := \alpha_{\mathbb{C}}(d)-\alpha_{\mathbb{C}}(d+1) > 0.
\end{equation*}
They supported the conjecture with extensive numerical experiments, exact formulas in finite dimensions, and asymptotic bounds~\cite[Section~4, especially equation~(16), p.~447, and Appendix~2, pp.~463--467]{bandeira2016approximating}.
The conjecture was confirmed very recently by Hutn\'{\i}k \cite[Theorem~8.1]{hutnik2026average}.
Although not stated explicitly there, the estimates in the proof also imply the following sharp leading behavior:\footnote{The auxiliary quantity used there to control the sign has the expansion $\frac{\log d}{12\pi d^{2}}+O(d^{-2})$.
Combining equations~(4.3) and~(5.9) and the identity used in the proof of Theorem~8.1 in~\cite{hutnik2026average} with $\Gamma(d+3/2)/d! \sim d^{1/2}$ gives the displayed formula.}
\begin{equation*}
\Delta_d = \frac{\log d}{8\pi d^3}+O(d^{-3}).
\end{equation*}
The present paper complements this leading relation by determining complete asymptotic expansions for $\Delta_d$ and $\alpha_{\mathbb{C}}(d)$ and by providing explicit upper and lower bounds for both quantities in every dimension.
Our approach takes a different route, beginning with a three-term recurrence for $Y_d = d^{3/2}\alpha_{\mathbb{C}}(d)$.

\subsection{\texorpdfstring
{Upper and lower bounds for $\Delta_{d}$ and $\alpha_{\mathbb{C}}(d)$}
{Upper and lower bounds for the decrement and average singular value}}

To present the results we need some notation.
For real $x \geq 1$, set
\begin{equation}
\ell_{x} := \log x+\gamma +6\log 2,
\label{eq:ell-x-intro}
\end{equation}
where $\log$ denotes the natural logarithm and $\gamma$ is Euler's constant,
\begin{equation}
\gamma
:= \lim_{m \to \infty}\left(
\sum_{k = 1}^{m}\frac{1}{k}-\log m
\right)
= 0.5772156649\ldots.
\label{eq:Euler-constant}
\end{equation}
For an integer $d \geq 1$, write $H_d := \sum_{k = 1}^{d}k^{-1}$.
We obtain the following bounds for every $d \geq 1$, which are the central result of this paper.

\begin{theorem}
\label{thm:main}
For every integer $d \geq 1$,
\begin{equation}
0
< \frac{\ell_{d}-10/3}{8\pi d^{3/2}(d+1)^{3/2}}
< \Delta_{d}
< \frac{H_d+6\log 2-10/3}{8\pi d^{3/2}(d+1)^{3/2}}.
\label{eq:main-bound}
\end{equation}
\end{theorem}

The difference between the upper and lower bounds in (\ref{eq:main-bound}) is itself bounded explicitly:
\begin{equation}
0
< \frac{H_d-\log d-\gamma}{8\pi d^{3/2}(d+1)^{3/2}}
< \frac{1}{16\pi d^{5/2}(d+1)^{3/2}}
< \frac{1}{16\pi d^4}.
\label{eq:delta-enclosure-width}
\end{equation}

In particular, $\alpha_{\mathbb{C}}(d)$ is strictly decreasing.
Both bounds in (\ref{eq:main-bound}) are asymptotic to $(\log d)/(8\pi d^{3})$, while (\ref{eq:delta-enclosure-width}) shows that their relative separation tends to zero.

Useful explicit bounds for $\alpha_{\mathbb{C}}\left(d\right)$ are harder to obtain directly.
In \cite[Theorem~7]{bandeira2016approximating} it was shown that $\frac{8}{3\pi}-\frac{5.05}{d} \leq \alpha_{\mathbb{C}}\left(d\right)$.
As observed in \cite{abreu2026recurrence}, since one can directly compute $\alpha_{\mathbb{C}}\left(1\right) = \sqrt{\frac{\pi}{4}}$, Hutn\'{\i}k's result together with (\ref{eq:alpha-limit}) implies, for all $d \geq 1$, the inequality
\begin{equation}
\frac{8}{3\pi} < \alpha_{\mathbb{C}}\left(d\right) \leq \sqrt{\frac{\pi}{4}}.
\label{sharpbounds}
\end{equation}
Although these endpoint bounds are optimal uniformly over $d \geq 1$, Theorem~\ref{thm:main} gives sharper bounds at each fixed dimension.
Set
\begin{equation}
\Lambda_{d}
:= \frac{1}{16\pi}
\left\{
\frac{\ell_{d}-10/3}{(d+1)^{2}}
+\log\left(1+\frac{1}{d}\right)
-\frac{1}{d+1}
\right\}
> 0.
\label{eq:Lambda-d}
\end{equation}
For the upper bound, set
\begin{equation}
\Omega_d
:= \frac{\ell_d-17/6}{16\pi d^2}
+\frac{3d+1}{32\pi d^3(2d+1)^2}.
\label{eq:Omega-d}
\end{equation}
This quantity is positive for every $d \geq 1$.
We then have the following improvement of (\ref{sharpbounds}).

\begin{corollary}
\label{CorAlpha}
For every integer $d \geq 1$,
\begin{equation}
\frac{8}{3\pi}+\Lambda_{d}
< \alpha_{\mathbb{C}}(d)
< \frac{8}{3\pi}+\Omega_d.
\label{CorAlphaBound}
\end{equation}
\end{corollary}

The bounds are derived from Theorem~\ref{thm:main} in Section~\ref{sec:consequences}.
Together, these bounds quantify the rate at which $\alpha_{\mathbb C}(d)$ approaches its limit~(\ref{eq:alpha-limit}) from above.
Figure~\ref{fig:quantitative-bounds} compares the exact quantities with these bounds and the leading asymptotic approximation.\footnote{At $d = 1$, the exact value $\alpha_{\mathbb C}(1) = \sqrt{\pi}/2$ is slightly sharper than the upper bound in Corollary~\ref{CorAlpha} and is used in panel~\protect\subref{fig:alpha-bounds}.}

\begin{figure}[t]
 \centering
 \begin{subfigure}[t]{0.49\textwidth}
  \centering
  \includegraphics[width=\linewidth]{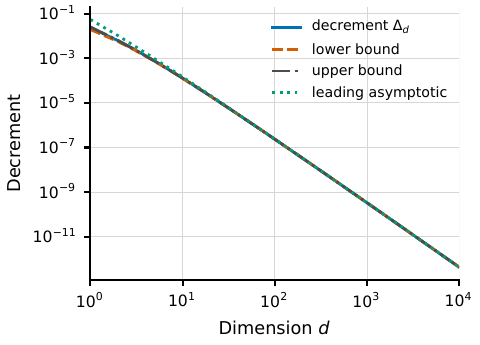}
  \caption{The decrement and its estimates.}
  \label{fig:decrement-bounds}
 \end{subfigure}
 \hfill
 \begin{subfigure}[t]{0.49\textwidth}
  \centering
  \includegraphics[width=\linewidth]{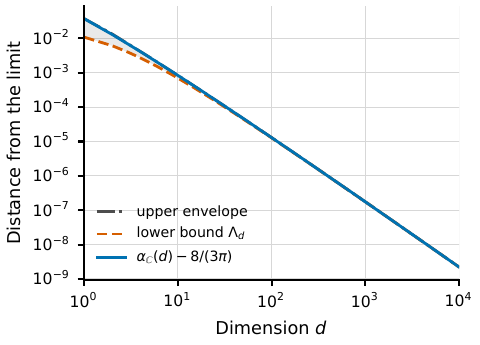}
  \caption{The envelope obtained from Corollary~\ref{CorAlpha} and the exact value at $d = 1$, after centering at the limit.}
  \label{fig:alpha-bounds}
 \end{subfigure}
 \caption{Numerical illustration of the non-asymptotic bounds for $1 \leq d \leq 10{,}000$.
 Panel~\protect\subref{fig:decrement-bounds} compares $\Delta_d$ with both bounds in Theorem~\ref{thm:main} and the leading asymptotic term $(\ell_d-10/3)/(8\pi d^3)$.
 Panel~\protect\subref{fig:alpha-bounds} shows $\alpha_{\mathbb C}(d)-8/(3\pi)$ inside the shaded envelope from Corollary~\ref{CorAlpha}, with the exact value used at $d = 1$.}
 \label{fig:quantitative-bounds}
\end{figure}

\subsection{Three-term recurrence}

Our results rely on the following connection with continuous dual Hahn polynomials.
Cunden, Mezzadri, O'Connell, and Simm identify the matrix moment underlying $\alpha_{\mathbb{C}}(d)$ with a continuous dual Hahn polynomial~\cite[Theorem~4.4 and equation~(4.12), p.~1107]{cunden2019moments}.
Together with the standard recurrence for this polynomial family, this identification gives a closed recurrence for the rescaled statistic
\begin{equation}
Y_{d} := d^{3/2}\alpha_{\mathbb{C}}(d),\qquad Y_{0} := 0.
\label{eq:Y-intro}
\end{equation}
We record the resulting recurrence here; Subsection~\ref{sec:dual-hahn-identification} gives the exact identification and normalization.

\begin{proposition}
\label{prop:structure}
The following three-term recurrence holds, with initial values $Y_{0} = 0$ and $Y_{1} = \sqrt{\pi}/2$:
\begin{equation}
Y_{d+1} = \left(2+\frac{3}{4d^{2}}\right) Y_{d}-Y_{d-1},\qquad d \geq 1.
\label{eq:Y-recurrence-intro}
\end{equation}
\end{proposition}
Once $Y_{0}$ and $Y_{1}$ are known, this recurrence computes the entire sequence and is used in both the asymptotic analysis and the bounds valid in every dimension.

\subsection{Complete asymptotic expansions}

The following result provides complete asymptotic expansions for $\alpha_{\mathbb{C}}(d)$ and $\Delta_d$, with all coefficients explicitly computable.
These expansions are an essential ingredient in the proof of Theorem~\ref{thm:main}.

\begin{theorem}
\label{thm:all-orders}
There exist uniquely determined sequences of recursively computable constants $(P_{j},Q_{j})_{j \geq 1}$ and $(C_{j},D_{j})_{j \geq 0}$, with the following properties.
For every integer $M \geq 0$, as $d \to \infty$,
\begin{equation}
\alpha_{\mathbb{C}}(d)
= \frac{8}{3\pi}
+\sum_{j = 1}^{M}\frac{P_{j}\log d+Q_{j}}{d^{2j}}
+O_{M}\!\left(\frac{\log d}{d^{2M+2}}\right).
\label{eq:alpha-all-orders}
\end{equation}
Moreover, for every integer $M \geq 0$, as $d \to \infty$,
\begin{equation}
\Delta_{d}
:= \alpha_{\mathbb{C}}(d)-\alpha_{\mathbb{C}}(d+1)
= \sum_{j = 0}^{M}\frac{C_{j}\log d+D_{j}}{d^{j+3}}
+O_{M}\!\left(\frac{\log d}{d^{M+4}}\right).
\label{eq:delta-all-orders}
\end{equation}
\end{theorem}

The proof of Theorem~\ref{thm:main} requires the coefficients in (\ref{eq:alpha-all-orders}) for $M = 1$ and those in (\ref{eq:delta-all-orders}) for $M = 0$.
Computing these constants gives, as $d \to \infty$,
\begin{align}
\alpha_{\mathbb{C}}(d)
& = \frac{8}{3\pi}
+\frac{\ell_{d}-17/6}{16\pi d^{2}}
+O\!\left(\frac{\log d}{d^{4}}\right),
\label{eq:alpha-leading-intro} \\
\Delta_{d}
& = \frac{\ell_{d}-10/3}{8\pi d^{3}}
+O\!\left(\frac{\log d}{d^{4}}\right).
\label{eq:delta-leading-intro}
\end{align}

\subsection{Application to the little Grothendieck problem over the unitary group}

Given a positive semidefinite block matrix $C \in \mathbb{C}^{dn\times dn}$, with blocks $C_{ij} \in \mathbb{C}^{d\times d}$, the problem is to find
\begin{equation}
\max_{U_{1},\ldots,U_{n} \in \mathcal{U}_{d}}
\sum_{i = 1}^{n}\sum_{j = 1}^{n}
\tr\left(C_{ij}^{H}U_{i}U_{j}^{H}\right).
\label{eq:unitary-little-Grothendieck}
\end{equation}

The variational representation (\ref{eq:alpha-variational}) comes in handy to explain the appearance of $\alpha_{\mathbb{C}}(d)$ in the little Grothendieck problem over the unitary group considered in \cite{bandeira2016approximating}, the context that originally motivated the conjecture.
Bandeira, Kennedy, and Singer first replace the unitary matrices in (\ref{eq:unitary-little-Grothendieck}) by matrices $Y_{i}$ satisfying $Y_{i}Y_{i}^{H} = I_{d}$.
They then multiply the $Y_{i}$ by a common normalized Gaussian matrix $R$ and replace each resulting square matrix $Y_{i}R$ by its unitary polar factor $\mathcal{P}(Y_{i}R)$.
Since $Y_{i}R$ has the same distribution as $G_{d}/\sqrt{d}$, formula (\ref{eq:alpha-variational}) shows that $\alpha_{\mathbb{C}}(d)$ measures the expected quality of this replacement.
They then use the positive semidefiniteness of $C$ to prove that the expected objective value of the rounded solution is at least $\alpha_{\mathbb{C}}(d)^{2}$ times the optimum.
Thus, the dimension dependence of $\alpha_{\mathbb{C}}(d)$ determines how this guarantee changes with $d$.
To rephrase our results in this setting, observe that using the lower bound in (\ref{CorAlphaBound}) twice, with $d$ and $d+1$, gives
\begin{equation*}
\alpha_{\mathbb{C}}(d)+\alpha_{\mathbb{C}}(d+1)
> \frac{16}{3\pi}+\Lambda_{d}+\Lambda_{d+1}.
\end{equation*}
The difference between successive approximation ratios factors as
\begin{equation*}
\alpha_{\mathbb{C}}(d)^{2}-\alpha_{\mathbb{C}}(d+1)^{2}
= \left(\alpha_{\mathbb{C}}(d)+\alpha_{\mathbb{C}}(d+1)\right) \Delta_{d}.
\end{equation*}
Combining this identity with (\ref{eq:main-bound}) and the preceding lower bound for the sum gives
\begin{equation}
\alpha_{\mathbb{C}}(d)^{2}-\alpha_{\mathbb{C}}(d+1)^{2}
> \left(\frac{16}{3\pi}+\Lambda_{d}+\Lambda_{d+1}\right)
\frac{\ell_{d}-10/3}{8\pi d^{3/2}(d+1)^{3/2}}
> 0.
\label{eq:cut-finite-decrement}
\end{equation}
The approximation ratio itself satisfies, from (\ref{CorAlphaBound}),
\begin{equation}
\left(\frac{8}{3\pi}+\Lambda_{d}\right) ^{2}
< \alpha_{\mathbb{C}}(d)^{2}
< \left(\frac{8}{3\pi}+\Omega_d\right)^2.
\label{eq:cut-improved-bounds}
\end{equation}

\subsection{Related work}
\label{sec:related-work}

Earlier work of the first author derived a relation for $\alpha_{\mathbb{C}}(d+1)-\alpha_{\mathbb{C}}(d)$ involving two explicit Laguerre integrals and isolated the sign condition that remained to be proved~\cite[Proposition~1 and equations~(1.6)--(1.9)]{abreu2026recurrence}.
Hutn{\'{\i}}k used this relation to prove strict decrease in the square case and, more generally, for $d\times(d+\lambda)$ complex Gaussian matrices with every fixed nonnegative integer $\lambda$~\cite[Theorem~8.1, pp.~10--11, and Theorem~B.4, pp.~27--28]{hutnik2026average}.

The present paper grew from a proof that $\alpha_{\mathbb{C}}(d)-\alpha_{\mathbb{C}}(d+1) > 0$ developed by the second author independently of Hutn{\'{\i}}k's manuscript and by different methods.
That proof used the recurrence in Proposition~\ref{prop:structure}, while the continuous dual Hahn identification of Cunden, Mezzadri, O'Connell, and Simm provides a conceptual source for the same recurrence.
Subsequent discussions led to the present joint work, which derives sharp upper and lower bounds valid in every dimension and complete asymptotic expansions from this structure.

While finalizing this paper, we learned of a very recent preprint by Baslingker and Dan~\cite{baslingker2026monotonicity}.
Their argument likewise specializes the general Laguerre moment recurrence associated with the continuous dual Hahn correspondence to give another proof of monotonicity, and it extends the result to every real nonnegative shape parameter.
That paper does not pursue the complete asymptotic expansions or the bounds valid in every dimension obtained here.

A separate recent preprint by Hutn{\'{\i}}k proves that the real average singular value increases with dimension~\cite{hutnik2026real}, as conjectured by Bandeira, Kennedy, and Singer~\cite{bandeira2016approximating}.

\subsection{Outline of the paper}

Section~\ref{sec:laguerre-recurrence} expresses the statistic through Laguerre moments, uses the continuous dual Hahn representation of these moments to obtain Proposition~\ref{prop:structure}, and derives the hypergeometric generating function underlying the asymptotic analysis.
Section~\ref{sec:asymptotic-expansion} expands this generating function at $z = 1$, extracts its coefficients, and uses the recurrence to explain why only even inverse powers occur in the expansion of $\alpha_{\mathbb{C}}(d)$.
Section~\ref{sec:sharp-decrement} then proves Theorem~\ref{thm:main}, and Section~\ref{sec:consequences} derives Corollary~\ref{CorAlpha} by summing the upper and lower bounds for the decrement.

\section{Laguerre moments: recurrence and generating function}

\label{sec:laguerre-recurrence}

We use the following convention for the generalized hypergeometric series with complex upper parameters $a_{1},\ldots,a_{p}$ and complex lower parameters $b_{1},\ldots,b_{q}$, where no lower parameter is a nonpositive integer \cite{AAR,szego1975orthogonal,nist2026digital}.
Here $(a)_{0} := 1$ and $(a)_{n} := a(a+1)\cdots(a+n-1)$ for $n \geq 1$:

\begin{equation*}
{}_{p}F_{q}\left(
\begin{array}{c}
a_{1},a_{2},...,a_{p} \\
b_{1},...,b_{q}
\end{array}
;z
\right)
= \sum_{n = 0}^{\infty}
\frac{(a_{1})_{n}(a_{2})_{n}\cdots(a_{p})_{n}}{(b_{1})_{n}(b_{2})_{n}\cdots(b_{q})_{n}}
\frac{z^{n}}{n!}.
\end{equation*}
For complex parameters $a,b,c$ such that $c$, $c-a$, and $c-b$ are not nonpositive integers and $\Real(c-a-b) > 0$, Gauss's summation formula \cite[DLMF 15.4.20]{nist2026digital} states
\begin{equation}
_{2}F_{1}\left(
\begin{array}{c}
a,b \\
c
\end{array}
;1
\right)
= \frac{\Gamma (c)\Gamma (c-a-b)}{\Gamma (c-a)\Gamma (c-b)}.
\label{Gauss}
\end{equation}
For $\alpha \geq 0$, the generalized Laguerre polynomial is defined by
\begin{equation*}
L_{n}^{(\alpha)}(x)
= \sum_{k = 0}^{n}(-1)^{k}\binom{n+\alpha}{n-k}\frac{x^{k}}{k!}.
\end{equation*}
With this normalization it satisfies the orthogonality relation
\begin{equation*}
\int_{0}^{\infty}x^{\alpha}e^{-x}L_{j}^{(\alpha)}(x)L_{k}^{(\alpha)}(x)\,dx
= \frac{\Gamma (k+\alpha +1)}{k!}\delta_{j,k}.
\end{equation*}

\subsection{Singular value density}

Let $L_{n} = L_{n}^{(0)}$ denote the ordinary Laguerre polynomial.
With the Gaussian normalization in (\ref{eq:def-alpha}), let $\lambda_{1},\ldots,\lambda_{d}$ denote the eigenvalues of $G_{d}G_{d}^{\ast}$.
They form the square Laguerre unitary ensemble with weight $e^{-x}$, whose standard kernel constructed from orthogonal polynomials gives the expected eigenvalue density~\cite[Section~3.2, pp.~90--97]{forrester2010log}
\begin{equation}
\rho_{d}(x) = e^{-x}\sum_{k = 0}^{d-1}L_{k}(x)^{2},\qquad x > 0.
\label{eq:lue-density}
\end{equation}
The normalization of $\rho_{d}$ means that
\begin{equation*}
\mathbb{E}\sum_{j = 1}^{d}f(\lambda_{j})
= \int_{0}^{\infty}f(x)\rho_{d}(x)\,dx
\end{equation*}
for every nonnegative Borel function $f$.
Since $\sigma_{j}(G_{d}/\sqrt{d}) = \sqrt{\lambda_{j}/d}$, the singular values are obtained from $f(x) = \sqrt{x}$.
Thus,
\begin{equation}
\alpha_{\mathbb{C}}(d) = \frac{Y_{d}}{d^{3/2}},
\qquad
\frac{Y_{d}}{\sqrt{\pi}} := \sum_{k = 0}^{d-1}u_{k},
\qquad
u_{k} := \frac{1}{\sqrt{\pi}}\int_{0}^{\infty}x^{1/2}e^{-x}L_{k}(x)^{2}\,dx.
\label{eq:alpha-Y}
\end{equation}
This is the representation used in~\cite[equation~(16), p.~447]{bandeira2016approximating} and~\cite[equations~(2.1)--(2.2)]{abreu2026recurrence}.

\subsection{Continuous dual Hahn identification and proof of Proposition~\protect\ref{prop:structure}}

\label{sec:dual-hahn-identification}
For a complex moment order with $\Real \mu > -1$, define the square Laguerre unitary ensemble moment
\begin{equation*}
\mathcal{Y}_{n}(\mu)
:= \mathbb{E}\operatorname{Tr}(G_{n}G_{n}^{\ast})^{\mu}
= \sum_{k = 0}^{n-1}\int_{0}^{\infty}x^{\mu}e^{-x}L_{k}(x)^{2}\,dx,
\qquad n \geq 1,
\end{equation*}
and set $\mathcal{Y}_{0}(\mu) := 0$.
Equation~(\ref{eq:alpha-Y}) gives $\mathcal{Y}_{n}(1/2) = Y_{n}$.

For $j \geq 0$, set
\begin{equation}
p_{j}(\mu)
:= {}_{3}F_{2}\!\left(
\begin{matrix}
-j, 1-\mu, 2+\mu \\
2, 2
\end{matrix}
;1
\right)
= \frac{\mathsf{S}_{j}^{\mathrm{CDH}}\!\left(-(\mu +1/2)^{2};3/2,1/2,1/2\right)}{(2)_{j}^{2}},
\label{eq:specialized-cdh}
\end{equation}
where the second equality is the continuous dual Hahn normalization in \cite[DLMF~18.26.2]{nist2026digital}.
In the notation of Cunden et al., $\mathcal{Y}_{n}(\mu) = Q_{\mu}^{\mathbb{C}}(n,n)$.
Their Theorem~4.4 and equation~(4.12), specialized to the square case, give
\begin{equation}
\mathcal{Y}_{n}(\mu) = n^{2}\Gamma (\mu +1)p_{n-1}(\mu),
\qquad \Real \mu > -1.
\label{eq:square-cdh-identification}
\end{equation}
The cited theorem obtains this identity for complex moment orders by Carlson continuation from the nonnegative integers~\cite[Theorem~4.4 and equation~(4.12), p.~1107]{cunden2019moments}.
Here $\mu$ determines the polynomial argument, while the three parameters of the continuous dual Hahn family are fixed at $3/2,1/2,1/2$.
For these positive parameters, the family is orthogonal with respect to its standard positive measure, supported on arguments $y^{2}$ with $y > 0$~\cite[DLMF Table~18.25.1 and equations~18.25.2, 18.25.6--18.25.8]{nist2026digital}.
At $\mu = 1/2$, the argument in (\ref{eq:specialized-cdh}) is $-1$, outside this support, but the recurrence used below remains valid there because it is a polynomial identity in the argument.

The continuous dual Hahn recurrence~\cite[Section~9.3, equation~(9.3.4)]{koekoek2010hypergeometric}, specialized to the parameters in (\ref{eq:specialized-cdh}), is
\begin{equation}
(n+1)^{2}p_{n}(\mu)
-\{2n^{2}+\mu (\mu +1)\}p_{n-1}(\mu)
+(n-1)^{2}p_{n-2}(\mu)
= 0,
\qquad n \geq 1.
\label{eq:specialized-cdh-recurrence}
\end{equation}
For $n = 1$, the last term vanishes, so no value of $p_{-1}$ is required.
Multiplying (\ref{eq:specialized-cdh-recurrence}) by $\Gamma (\mu +1)$ and using (\ref{eq:square-cdh-identification}) gives
\begin{equation}
\mathcal{Y}_{n+1}(\mu)-2\mathcal{Y}_{n}(\mu)+\mathcal{Y}_{n-1}(\mu)
= \frac{\mu (\mu +1)}{n^{2}}\mathcal{Y}_{n}(\mu),
\qquad n \geq 1.
\label{eq:general-Y-recurrence}
\end{equation}
At $\mu = 1/2$, equation~(\ref{eq:alpha-Y}) gives $\mathcal{Y}_{n}(1/2) = Y_{n}$, so (\ref{eq:general-Y-recurrence}) becomes (\ref{eq:Y-recurrence-intro}).
The initial values follow from $Y_{0} = 0$ and $\mathcal{Y}_{1}(1/2) = \Gamma (3/2) = \sqrt{\pi}/2$.
This proves Proposition~\ref{prop:structure}.
For comparison, the related but distinct associated continuous Hahn family and its recurrence and orthogonality theory were studied by Gupta, Ismail, and Masson~\cite{gupta1991associated}.

\subsection{Laguerre moment generating function}

We can easily obtain the generating function needed for the asymptotic analysis from the Laguerre representation.
The Laguerre connection formula specialized to the parameters used here is~\cite[DLMF~18.18.18]{nist2026digital}
\begin{equation}
L_{n}(x)
= \sum_{k = 0}^{n}\frac{\left(-1/2\right)_{n-k}}{(n-k)!}L_{k}^{(1/2)}(x).
\label{eq:laguerre-connection}
\end{equation}
The polynomials on the right are orthogonal for the weight $x^{1/2}e^{-x}$, with~\cite[DLMF Table~18.3.1]{nist2026digital}
\begin{equation}
\int_{0}^{\infty}x^{1/2}e^{-x}L_{j}^{(1/2)}(x)L_{k}^{(1/2)}(x)\,dx
= \frac{\Gamma (k+3/2)}{k!}\delta_{j,k}.
\label{eq:laguerre-half-orthogonality}
\end{equation}
Squaring (\ref{eq:laguerre-connection}) and applying (\ref{eq:laguerre-half-orthogonality}) gives
\begin{equation}
u_{n}
= \frac{1}{\sqrt{\pi}}\int_{0}^{\infty}x^{1/2}e^{-x}L_{n}(x)^{2}\,dx
= \sum_{k = 0}^{n}
\left(\frac{\left(-1/2\right)_{n-k}}{(n-k)!}\right) ^{2}
\frac{\Gamma (k+3/2)}{\sqrt{\pi}\,k!}.
\label{eq:u-convolution}
\end{equation}
The squared connection coefficients and the Gamma factors have generating functions ${}_{2}F_{1}(-1/2,-1/2;1;z)$ and $\frac{1}{2}(1-z)^{-3/2}$, respectively.
Consequently, for $|z| < 1$,
\begin{equation}
U(z)
:= \sum_{n = 0}^{\infty}u_{n}z^{n}
= \frac{1}{2}(1-z)^{-3/2}{}_{2}F_{1}\!\left(-\frac{1}{2},-\frac{1}{2};1;z\right).
\label{eq:U-generating}
\end{equation}

\section{Proof of Theorem~\protect\ref{thm:all-orders}}

\label{sec:asymptotic-expansion}

The proof of Theorem~\ref{thm:all-orders} proceeds according to the following outline.
We first derive the generating function of $Y_{n}$, which can be explicitly written in terms of a hypergeometric function ${}_{2}F_{1}$ times a factor $\sqrt{\pi}/2z(1-z)^{-5/2}$.
This generating function is then expanded at its singularity $z = 1$, setting $t = 1-z$, and all the corresponding coefficients are obtained explicitly.
We then substitute this expansion into the generating function and use $z = 1-(1-z)$ to obtain the singular expansion of the full generating function in powers of $\left(1-z\right)$.
We then extract the coefficients of this singular expansion term by term to obtain the asymptotic expansion of $Y_{n}$.
The recurrence (\ref{eq:Y-recurrence-intro}) is then used to show that the odd powers of the expansion vanish.
The expansion of $\Delta_{n}$ is then obtained by subtracting the expansions at consecutive indices $n$ and $n+1$.

Since
\begin{equation*}
\frac{1}{\sqrt{\pi}}\sum_{n = 1}^{\infty}Y_{n}z^{n}
= \frac{z}{1-z}\sum_{n = 0}^{\infty}u_{n}z^{n},
\end{equation*}
we obtain from (\ref{eq:U-generating}) the generating function of $Y_{n}$:
\begin{equation}
\sum_{n = 1}^{\infty}Y_{n}z^{n}
= \sqrt{\pi}\frac{z}{1-z}U(z)
= \frac{\sqrt{\pi}}{2}z(1-z)^{-5/2}F(z),
\label{eq:Y-generating}
\end{equation}
where $F(z) := {}_{2}F_{1}\left(-\frac{1}{2},-\frac{1}{2};1;z\right)$.

\subsection{\texorpdfstring
{Expansion near $z = 1$}
{Expansion near z = 1}}

Formula~\cite[DLMF 15.8.10]{nist2026digital}, with $a = b = -1/2$ and $m = 2$, gives, for $0 < |t| < 1$ and $|\arg t| < \pi$, and $z = 1-t$, the $_{2}F_{1}$ hypergeometric expansion:
\begin{align}
F(1-t)
= {}& \frac{4}{\pi}-\frac{t}{\pi}  \notag \\
& -\frac{t^{2}}{4\pi}
\sum_{k = 0}^{\infty}
\frac{(3/2)_{k}^{2}}{k!(k+2)!}t^{k}
\left[\log t-\psi (k+1)-\psi (k+3)+2\psi (k+3/2)\right].
\label{eq:F-continuation-specialized}
\end{align}
Here $\psi = \Gamma ^{\prime}/\Gamma$ is the digamma function.
The normalized hypergeometric function used in the cited formula coincides with $F$ because $\Gamma (1) = 1$.
Writing (\ref{eq:F-continuation-specialized}) as
\begin{equation}
F(1-t)
= \sum_{m = 0}^{\infty}a_{m}t^{m}
+\log t\sum_{m = 2}^{\infty}b_{m}t^{m},
\label{eq:F-frobenius}
\end{equation}
we obtain
\begin{equation}
a_{0} = \frac{4}{\pi},\qquad a_{1} = -\frac{1}{\pi},
\label{eq:F-first-two-coefficients}
\end{equation}
and, for every integer $k \geq 0$,
\begin{align}
b_{k+2}& = -\frac{1}{4\pi}\frac{(3/2)_{k}^{2}}{k!(k+2)!},
\label{eq:b-closed} \\
a_{k+2}& = \frac{1}{4\pi}
\frac{(3/2)_{k}^{2}}{k!(k+2)!}
\left[\psi (k+1)+\psi (k+3)-2\psi (k+3/2)\right].
\label{eq:a-closed}
\end{align}
In particular,
\begin{equation}
a_{2} = \frac{8\log 2-5}{16\pi},\qquad b_{2} = -\frac{1}{8\pi}.
\label{eq:frobenius-initial}
\end{equation}
Thus all the coefficients in the expansion at $z = 1$ are explicit.

\subsection{Truncation and coefficient extraction}

\label{subsec:coefficient-extraction}

Define
\begin{equation*}
H(z) := \frac{z}{2}(1-z)^{-5/2}F(z).
\end{equation*}
By (\ref{eq:Y-generating}),
\begin{equation}
H(z)
= \frac{1}{\sqrt{\pi}}\sum_{n = 1}^{\infty}Y_{n}z^{n},
\qquad |z| < 1.
\label{H-Y}
\end{equation}
Set
\begin{equation*}
a_{-1} = b_{-1} = b_{0} = b_{1} = 0
\end{equation*}
and, for $j \geq 0$, define
\begin{equation}
c_{j} := \frac{a_{j}-a_{j-1}}{2},
\qquad
d_{j} := \frac{b_{j}-b_{j-1}}{2}.
\label{eq:c-d-def}
\end{equation}
Writing $z = 1-(1-z)$, the expansion (\ref{eq:F-frobenius}) gives
\begin{equation}
H(z)
= \sum_{j = 0}^{\infty}
\left(c_{j}+d_{j}\log (1-z)\right)
(1-z)^{j-5/2}.
\label{eq:H-singular-expansion}
\end{equation}
Consider the truncated expansion
\begin{equation*}
H_{J}(z)
= \sum_{j = 0}^{J}
\left(c_{j}+d_{j}\log (1-z)\right)
(1-z)^{j-5/2}.
\end{equation*}
From now on, denote by $[z^{n}]\left\{E(z)\right\}$ the coefficient of $z^{n}$ in the expansion of $E(z).$
Using this notation, (\ref{H-Y}) gives $[z^{n}]\left\{H(z)\right\} = Y_{n}/\sqrt{\pi}$.
Thus, $Y_{n}^{(J)}$, the approximation to $Y_{n}$ obtained by retaining the singular terms with $0 \leq j \leq J$, can be naturally defined as:
\begin{equation*}
Y_{n}^{(J)} := {}\sqrt{\pi}[z^{n}]\left\{H_{J}(z)\right\},
\end{equation*}
while $\alpha_{\mathbb{C}}^{(J)}(n)$, the corresponding approximation to $\alpha_{\mathbb{C}}(n)$, is defined as:
\begin{equation}
\alpha_{\mathbb{C}}^{(J)}(n) := \frac{Y_{n}^{(J)}}{n^{3/2}}.
\label{eq:alpha-J}
\end{equation}
For $\lambda \notin \{0,1,2,\ldots\}$, we have
\begin{equation*}
(1-z)^{\lambda}
= \sum\limits_{n = 0}^{\infty}
\frac{\Gamma (n-\lambda)}{\Gamma (-\lambda)\Gamma (n+1)}z^{n},
\qquad
\left\vert z\right\vert < 1.
\end{equation*}
Differenting this equality with respect to $\lambda$ gives, since $\psi = \frac{\Gamma ^{\prime}}{\Gamma}$,
\begin{equation*}
(1-z)^{\lambda}\log (1-z)
= \sum\limits_{n = 0}^{\infty}
\frac{\Gamma (n-\lambda)}{\Gamma (-\lambda)\Gamma (n+1)}
\left[\psi (-\lambda)-\psi (n-\lambda)\right].
\end{equation*}
Thus:
\begin{align}
\lbrack z^{n}](1-z)^{\lambda}
& = \frac{\Gamma (n-\lambda)}{\Gamma (-\lambda)\Gamma (n+1)},
\label{eq:power-coefficient} \\
\lbrack z^{n}](1-z)^{\lambda}\log (1-z)
& = \frac{\Gamma (n-\lambda)}{\Gamma (-\lambda)\Gamma (n+1)}
\left[\psi (-\lambda)-\psi (n-\lambda)\right].
\label{eq:log-power-coefficient}
\end{align}
Applying these identities with $\lambda = j-5/2$ to the terms of (\ref{eq:H-singular-expansion}) with $0 \leq j \leq J$ gives
\begin{align}
Y_{n}^{(J)}
:= {}& \sqrt{\pi}
\sum_{j = 0}^{J}
c_{j}
\frac{\Gamma (n-j+5/2)}{\Gamma (5/2-j)\Gamma (n+1)}  \notag \\
& +\sqrt{\pi}
\sum_{j = 2}^{J}
d_{j}
\frac{\Gamma (n-j+5/2)}{\Gamma (5/2-j)\Gamma (n+1)}
\left[\psi (5/2-j)-\psi (n-j+5/2)\right].
\label{eq:Y-J}
\end{align}
Apart from $\Gamma (n+1)$, all Gamma and digamma arguments in (\ref{eq:Y-J}) are nonintegral half-integers.
Thus $Y_{n}^{(J)}$ is well defined even when $J > n$.
We now estimate the error resulting by truncating (\ref{eq:H-singular-expansion}).
Let
\begin{equation*}
R_{J}(z)
:= H(z)
-\sum_{j = 0}^{J}
\left(c_{j}+d_{j}\log (1-z)\right)
(1-z)^{j-5/2}.
\end{equation*}
The convergent expansion (\ref{eq:F-frobenius}) gives, uniformly in closed subsectors of $|\arg (1-z)| < \pi$,
\begin{equation}
R_{J}(z)
= O_{J}\left(
|1-z|^{J-3/2}
\left(1+|\log (1-z)|\right)
\right),
\label{ORJ}
\end{equation}
as $z \to 1$.
The principal branch of $F$ is analytic in $\mathbb{C}\setminus \lbrack 1,\infty)$.
The finite sum defining $H_{J}$ has the same property, and hence $R_{J} = H-H_{J}$ is analytic in a $\Delta$-domain at $z = 1$.
The transfer theorem for algebraic and logarithmic singularities tells that the singular behavior of a generating function near its dominant singularity determines the asymptotic behavior of its coefficients \cite[Theorem VI.3, p.~390]{flajolet2009analytic}.
Thus, (\ref{ORJ}) gives
\begin{equation*}
\lbrack z^{n}]\left\{R_{J}(z)\right\} = O_{J}\left(n^{-J+1/2}\log n\right).
\end{equation*}
Since $[z^{n}]\left\{H(z)\right\} = Y_{n}/\sqrt{\pi}$ and $[z^{n}]\left\{H_{J}(z)\right\} = Y_{n}^{(J)}/\sqrt{\pi}$, we obtain
\begin{align}
Y_{n}& = Y_{n}^{(J)}+O_{J}\left(n^{-J+1/2}\log n\right),
\label{eq:Y-transfer} \\
\alpha_{\mathbb{C}}(n)& = \alpha_{\mathbb{C}}^{(J)}(n)
+O_{J}\left(\frac{\log n}{n^{J+1}}\right).
\label{eq:alpha-transfer}
\end{align}

\subsection{Proof of (\protect\ref{eq:alpha-all-orders}) and (\protect\ref{eq:delta-all-orders})}

\emph{Step 1.}
Using \cite[DLMF~5.11.2 and 5.11.13]{nist2026digital}, we observe that, for each fixed $j$, the Gamma quotient in (\ref{eq:Y-J}) has an asymptotic expansion in descending powers of $n$, while the corresponding digamma difference equals $-\log n+O(1)$.
From \cite[DLMF~5.11.8 and 5.11.17]{nist2026digital} it follows that the corresponding coefficients are given explicitly in terms of Bernoulli polynomials.
Thus, $\alpha_{\mathbb{C}}^{(J)}(n)$ has an expansion to any prescribed order with coefficients of the form $A\log n+B$.

\emph{Step 2.}
We show that every odd inverse power vanishes.
By Step 1 and (\ref{eq:alpha-transfer}), $\alpha_{\mathbb{C}}(n)$ has, to arbitrary order, an expansion in integer powers of $n^{-1}$, whose coefficients have the form $A\log n+B$.
Write this expansion as
\begin{equation*}
\alpha_{\mathbb{C}}(n)
\sim \sum_{j \geq 0}(A_{j}\log n+B_{j})n^{-j}.
\end{equation*}
Since $Y_{n} = n^{3/2}\alpha_{\mathbb{C}}(n)$, the corresponding expansion of $Y_{n}$ is
\begin{equation}
Y_{n}
\sim \sum_{j \geq 0}n^{3/2-j}(A_{j}\log n+B_{j}).
\label{Yn}
\end{equation}
To compute the action of the three-term recurrence operator
\begin{equation}
\mathcal{L}_{n}(f) := f(n+1)-2f(n)+f(n-1)
\label{recOper}
\end{equation}
on each term of the expansion of $Y_{n}$, write
\begin{equation*}
f(n) = n^{p}(A\log n+B).
\end{equation*}
Expanding $f(n+1)$ and $f(n-1)$ in Taylor series and adding the results, gives
\begin{equation}
f(n+1)-2f(n)+f(n-1)
= 2\sum\limits_{k = 1}^{\infty}\frac{f^{(2k)}(n)}{(2k)!}.
\label{Taylor}
\end{equation}
Since $f^{(4)}(n) = O\left(n^{p-4}\log n\right)$ and
\begin{equation*}
f^{\prime \prime}(n)
= n^{p-2}\left(p(p-1)(A\log n+B)+(2p-1)A\right),
\end{equation*}
subtracting to both sides the term
\begin{equation*}
\frac{3}{4n^{2}}f(n)
= \frac{3}{4}n^{p-2}\left(A\log n+B\right),
\end{equation*}
leads to
\begin{align}
& f(n+1)-2f(n)+f(n-1)-\frac{3}{4n^{2}}f(n)  \notag \\
& \quad = {}n^{p-2}
\left\{
\left(p(p-1)-\frac{3}{4}\right) (A\log n+B)
+(2p-1)A
\right\}
+O\!\left(n^{p-4}\log n\right).
\label{right}
\end{align}
We now substitute the expansion (\ref{Yn}) into (\ref{eq:Y-recurrence-intro}) and compare, at each power of $n$, the logarithmic and constant coefficients.
At the order $n^{-j-1/2}$, formula (\ref{Taylor}) shows that the possible contributions come from the term of index $j$ through its second derivative, from the term of index $j-2$ through its fourth derivative, from the term of index $j-4$ through its sixth derivative, and so on.
Thus the equation at index $j$ involves only the coefficients at indices

\begin{equation*}
j, j-2, j-4,\ldots.
\end{equation*}
Consequently, the equations with even indices involve only coefficients with even indices, while those with odd indices involve only coefficients with odd indices.
To complete the proof, we show by induction that $A_{j} = B_{j} = 0$ for every odd index $j$ in the expansion (\ref{Yn}).
For the initial case $j = 1$, only the term indexed by $1$ contributes at order $n^{-3/2}$.
Taking $p = 1/2$, $A = A_{1}$, and $B = B_{1}$ in (\ref{right}), and comparing the logarithmic and constant coefficients at order $n^{-3/2}$ in (\ref{eq:Y-recurrence-intro}), gives
\begin{equation*}
A_{1} = 0,\qquad B_{1} = 0.
\end{equation*}
Now let $j \geq 3$ be odd and assume, as the induction hypothesis, that all coefficients at smaller odd indices vanish; that is,
\begin{equation*}
A_{1} = B_{1} = A_{3} = B_{3} = \cdots = A_{j-2} = B_{j-2} = 0.
\end{equation*}
At order $n^{-j-1/2}$, the recurrence involves only the coefficients at indices
\begin{equation*}
j, j-2, j-4,\ldots,1.
\end{equation*}
By the induction hypothesis, all contributions from the indices $j-2,j-4,\ldots ,1$ vanish.
It remains only to consider the term with index $j$.
Taking
\begin{equation*}
p = \frac{3}{2}-j,\qquad A = A_{j},\qquad B = B_{j}
\end{equation*}
in (\ref{right}), the multiplier of $A_{j}\log n+B_{j}$ is
\begin{equation*}
\left(\frac{3}{2}-j\right) \left(\frac{1}{2}-j\right) -\frac{3}{4}
= j(j-2).
\end{equation*}
Comparison of the logarithmic coefficients at order $n^{-j-1/2}$ therefore gives
\begin{equation*}
j(j-2)A_{j} = 0.
\end{equation*}
Since $j$ is odd, $j(j-2) \neq 0$, and hence $A_{j} = 0$.
Comparison of the constant coefficients at the same order gives
\begin{equation*}
j(j-2)B_{j}+2(1-j)A_{j} = 0.
\end{equation*}
Since $A_{j} = 0$ and $j(j-2) \neq 0$, it follows that $B_{j} = 0$.
This completes the induction and proves that all odd inverse powers in the asymptotic expansion of $\alpha_{\mathbb{C}}(n)$ vanish.

\emph{Step 3.}
To prove (\ref{eq:alpha-all-orders}), take $J = 2$ when $M = 0$ and $J = 2M+1$ when $M \geq 1$, expand $\alpha_{\mathbb{C}}^{(J)}(n)$ through order $n^{-2M}$, and combine the finite expansion with (\ref{eq:alpha-transfer}).
The constant term is $8/(3\pi)$ by (\ref{eq:alpha-limit}); relabeling the coefficients with even indices proves (\ref{eq:alpha-all-orders}).

The constant term in the expansion of $\alpha_{\mathbb{C}}(n)$ cancels in a consecutive difference, while the first correction, of order $n^{-2}\log n$, gives a difference of order $n^{-3}\log n$.
This explains the starting power in (\ref{eq:delta-all-orders}).
To obtain that expansion with a controlled remainder, take $J = M+3$ in (\ref{eq:alpha-transfer}), write the resulting finite approximations at $n$ and $n+1$, and subtract.
Bounding the two remainders separately gives
\begin{equation*}
\Delta_{n}
= \alpha_{\mathbb{C}}^{(J)}(n)
-\alpha_{\mathbb{C}}^{(J)}(n+1)
+O_{M}\!\left(\frac{\log n}{n^{M+4}}\right).
\end{equation*}
Expanding this finite difference proves (\ref{eq:delta-all-orders}); no cancellation between the remainder terms is required.

\subsection{Leading asymptotic coefficients}

We now compute the first correction to the limiting value $8/(3\pi)$.
After division of (\ref{eq:Y-J}) by $n^{3/2}$, the term indexed by $j$ begins at the inverse power $n^{-j}$, with an additional factor $\log n$ when $d_{j} \neq 0$.
Consequently, to determine all terms with inverse powers $n^{0}$, $n^{-1}$, and $n^{-2}$, including their logarithmic factors, it suffices to retain the terms indexed by $j = 0,1,2$.
We denote their respective contributions to $\alpha_{\mathbb{C}}^{(J)}(n)$ ($J \geq 2$) by $T_{0}(n)$, $T_{1}(n)$ and $T_{2}(n)$.
From (\ref{eq:frobenius-initial}) and (\ref{eq:c-d-def}), the required coefficients in the expansion (\ref{eq:H-singular-expansion}) are

\begin{equation*}
c_{0} = \frac{2}{\pi},
c_{1} = -\frac{5}{2\pi},
c_{2} = \frac{1}{\pi}\left(\frac{\log 2}{4}+\frac{11}{32}\right);
\qquad
d_{0} = d_{1} = 0,
d_{2} = -\frac{1}{16\pi}.
\end{equation*}
The $j = 0$ contribution is

\begin{equation*}
T_{0}(n)
:= \frac{\sqrt{\pi}c_{0}}{n^{3/2}}
\frac{\Gamma (n+5/2)}{\Gamma (5/2)\Gamma (n+1)}.
\end{equation*}
Using
\begin{equation*}
\frac{\Gamma (n+5/2)}{\Gamma (n+1)n^{3/2}}
= 1+\frac{15}{8n}+\frac{65}{128n^{2}}+O(n^{-3}),
\end{equation*}
we obtain
\begin{equation*}
T_{0}(n)
= \frac{8}{3\pi}+\frac{5}{\pi n}+\frac{65}{48\pi n^{2}}+O(n^{-3}).
\end{equation*}
Similarly, the $j = 1$ contribution is
\begin{equation*}
T_{1}(n)
:= \frac{\sqrt{\pi}c_{1}}{n^{3/2}}
\frac{\Gamma (n+3/2)}{\Gamma (3/2)\Gamma (n+1)}.
\end{equation*}
Since
\begin{equation*}
\frac{\Gamma (n+3/2)}{\Gamma (n+1)n^{3/2}}
= \frac{1}{n}\left(1+\frac{3}{8n}+O(n^{-2})\right),
\end{equation*}
it follows that
\begin{equation*}
T_{1}(n)
= -\frac{5}{\pi n}-\frac{15}{8\pi n^{2}}+O(n^{-3}).
\end{equation*}
For $j = 2$, the analytic and logarithmic parts must be considered together.
Their sum is
\begin{equation*}
T_{2}(n)
:= \frac{1}{n^{3/2}}
\frac{\Gamma (n+1/2)}{\Gamma (n+1)}
\left\{
c_{2}+d_{2}\left[\psi (1/2)-\psi (n+1/2)\right]
\right\}.
\end{equation*}
Using
\begin{equation*}
\frac{\Gamma (n+1/2)}{\Gamma (n+1)n^{3/2}}
= \frac{1}{n^{2}}\left(1+O(n^{-1})\right)
\end{equation*}
and
\begin{equation*}
\psi (1/2) = -\gamma -2\log 2,
\qquad
\psi (n+1/2) = \log n+O(n^{-1}),
\end{equation*}
we find
\begin{equation*}
T_{2}(n)
= \frac{1}{16\pi n^{2}}\left(\ell_{n}+\frac{11}{2}\right)
+O\left(\frac{\log n}{n^{3}}\right),
\end{equation*}
where $\ell_{n} = \log n+\gamma +6\log 2$.
Adding the three contributions gives
\begin{equation}
\alpha_{\mathbb{C}}(n)
= \frac{8}{3\pi}
+\frac{\ell_{n}-17/6}{16\pi n^{2}}
+O\left(\frac{\log n}{n^{4}}\right).
\label{eq:alpha-asymptotic}
\end{equation}
Finally,
\begin{equation*}
\frac{\ell_{n}-a}{n^{2}}
-\frac{\ell_{n+1}-a}{(n+1)^{2}}
= \frac{2(\ell_{n}-a)-1}{n^{3}}
+O\left(\frac{\log n}{n^{4}}\right).
\end{equation*}
Taking $a = 17/6$ and subtracting the expansion at $n+1$ from the expansion at $n$ gives
\begin{equation}
\Delta_{n}
= \frac{\ell_{n}-10/3}{8\pi n^{3}}
+O\left(\frac{\log n}{n^{4}}\right).
\label{eq:delta-asymptotic}
\end{equation}

\section{Proof of Theorem~\protect\ref{thm:main}}

\label{sec:sharp-decrement}

The proof of Theorem \ref{thm:main} combines the recurrence (\ref{eq:Y-recurrence-intro}) with the leading asymptotic formula (\ref{eq:delta-asymptotic}).
The recurrence will show that a suitable auxiliary sequence is strictly decreasing, while the asymptotic formula identifies its limit.
By writing $\Delta_{d}$ in terms of $Y_{d}$ as
\begin{equation*}
\Delta_{d}
= \alpha_{\mathbb{C}}(d)-\alpha_{\mathbb{C}}(d+1)
= \frac{Y_{d}}{d^{3/2}}-\frac{Y_{d+1}}{(d+1)^{3/2}},
\end{equation*}
we obtain:
\begin{align*}
& \frac{(d+1)^{3/2}(d+2)^{3/2}}{\sqrt{\pi}}\Delta_{d+1}
-\frac{d^{3/2}(d+1)^{3/2}}{\sqrt{\pi}}\Delta_{d} \\
& \quad = \frac{1}{\sqrt{\pi}}
\left\{
\left((d+2)^{3/2}+d^{3/2}\right)Y_{d+1}
-(d+1)^{3/2}(Y_{d+2}+Y_{d})
\right\}.
\end{align*}
Replacing $d$ by $d+1$ in the recurrence (\ref{eq:Y-recurrence-intro}) gives
\begin{equation*}
Y_{d+2}+Y_{d}
= \left(2+\frac{3}{4(d+1)^{2}}\right) Y_{d+1}.
\end{equation*}
Consequently,
\begin{align}
& \frac{(d+1)^{3/2}(d+2)^{3/2}}{\sqrt{\pi}}\Delta_{d+1}
-\frac{d^{3/2}(d+1)^{3/2}}{\sqrt{\pi}}\Delta_{d}  \notag \\
& \qquad = \frac{Y_{d+1}}{\sqrt{\pi}}
\left[
(d+2)^{3/2}-2(d+1)^{3/2}+d^{3/2}
-\frac{3}{4\sqrt{d+1}}
\right].
\label{eq:scaled-delta-increment}
\end{align}
The expression in brackets is positive.
Indeed,
\begin{align*}
& (d+2)^{3/2}-2(d+1)^{3/2}+d^{3/2} \\
& \qquad = \int_{-1}^{1}(1-|t|)\frac{3}{4\sqrt{d+1+t}}\,dt.
\end{align*}
The weight $1-|t|$ has integral one and mean zero on $[-1,1]$, and $t \mapsto (d+1+t)^{-1/2}$ is strictly convex.
Jensen's inequality therefore gives
\begin{equation*}
(d+2)^{3/2}-2(d+1)^{3/2}+d^{3/2} > \frac{3}{4\sqrt{d+1}}.
\end{equation*}
Since $Y_{d+1} > 0$, (\ref{eq:scaled-delta-increment}) implies
\begin{equation*}
\frac{(d+1)^{3/2}(d+2)^{3/2}}{\sqrt{\pi}}\Delta_{d+1}
> \frac{d^{3/2}(d+1)^{3/2}}{\sqrt{\pi}}\Delta_{d}.
\end{equation*}
Before completing the proof of Theorem~\ref{thm:main}, we need two auxiliary results.
The first is an upper bound for $\alpha_{\mathbb{C}}(d)$.

\begin{lemma}
\label{lem:alpha-upper}
For every integer $d \geq 1$, set
\begin{equation*}
B_d
:= \frac{8}{3\pi}
\frac{(d+1/2)\Gamma(d+1/2)}{d^{3/2}\Gamma(d)}.
\end{equation*}
Then
\begin{equation}
\alpha_{\mathbb{C}}(d)
< B_d
< \frac{8}{3\pi}\frac{d+1/2}{d}.
\label{eq:alpha-upper}
\end{equation}
\end{lemma}

\begin{proof}
Write
\begin{equation*}
F(z)
= {}_{2}F_{1}\!\left(-\frac{1}{2},-\frac{1}{2};1;z\right)
= \sum_{m = 0}^{\infty}\frac{(-1/2)_{m}^{2}}{(m!)^{2}}z^{m}.
\end{equation*}
Multiplying the binomial expansion,
\begin{equation*}
(1-z)^{-3/2}
= \sum_{k = 0}^{\infty}\frac{(3/2)_{k}}{k!}z^{k},
\end{equation*}
by $F(z)$, and using the generating function (\ref{eq:U-generating}) of the sequence $\{u_{k}\}$ to compare the coefficients of $z^{k}$, gives
\begin{equation*}
u_{k}
= \frac{1}{2}\sum_{m = 0}^{k}
\frac{(-1/2)_{m}^{2}}{(m!)^{2}}
\frac{(3/2)_{k-m}}{(k-m)!}.
\end{equation*}
Since $(3/2)_{k}/k!$ is increasing, the power series coefficients of $F(z)$ are positive, and extending the finite sum over $m$ to the infinite sum adds a nonzero positive tail,
\begin{equation*}
u_{k}
< \frac{1}{2}\frac{(3/2)_{k}}{k!}
\sum_{m = 0}^{\infty}
\frac{(-1/2)_{m}^{2}}{(m!)^{2}}
= \frac{1}{2}\frac{(3/2)_{k}}{k!}F(1)
= \frac{2}{\pi}\frac{(3/2)_{k}}{k!},
\end{equation*}
where in the last identity Gauss's summation formula (\ref{Gauss}) was used to obtain $F(1) = \frac{4}{\pi}$.
Using the identity
\begin{equation*}
\sum_{k = 0}^{n}\frac{(3/2)_{k}}{k!}
= 1+\sum_{k = 0}^{n}
\left(
\frac{(5/2)_{k}}{k!}
-\frac{(5/2)_{k-1}}{(k-1)!}
\right)
= \frac{(5/2)_{n}}{n!},
\end{equation*}
together with the definition of $Y_{n+1}$ (\ref{eq:alpha-Y}), gives
\begin{equation*}
\frac{Y_{n+1}}{\sqrt{\pi}}
= \sum_{k = 0}^{n}u_{k}
< \sum_{k = 0}^{n}\frac{2}{\pi}\frac{(3/2)_{k}}{k!}
= \frac{2}{\pi}\frac{(5/2)_{n}}{n!}.
\end{equation*}
Setting $d = n+1$ and writing the factorials in terms of the Gamma function gives
\begin{equation*}
\alpha_{\mathbb{C}}(d)
= \frac{Y_{d}}{d^{3/2}}
< B_d.
\end{equation*}
By the strict log convexity of the Gamma function (or by using Cauchy Schwarz in the integral representation of $\Gamma (d+1/2)$),
\begin{equation*}
\Gamma (d+1/2)^{2}
= \Gamma \left(\frac{d+(d+1)}{2}\right) ^{2}
< \Gamma(d)\Gamma(d+1)
= d\Gamma (d)^{2}.
\end{equation*}
Thus $\Gamma (d+1/2) < \sqrt{d}\,\Gamma (d)$, which proves the second inequality in (\ref{eq:alpha-upper}).
\end{proof}

Finally, we will need an upper bound for the positive quantity appearing in (\ref{eq:scaled-delta-increment}):
\begin{equation}
q_{d}
:= (d+2)^{3/2}-2(d+1)^{3/2}+d^{3/2}
-\frac{3}{4\sqrt{d+1}}.
\label{qd}
\end{equation}

\begin{lemma}
\label{lem:q-upper}
For every integer $d \geq 2$,
\begin{equation}
q_{d}
< \frac{3}{64\sqrt{d+1}\,(d+3/2)}\log \frac{d+1}{d}.
\label{eq:q-upper}
\end{equation}
\end{lemma}

\begin{proof}
Writing (\ref{qd}) as
\begin{equation*}
q_{d}
= (d+1)^{3/2}
\left[
\left(1+\frac{1}{d+1}\right) ^{3/2}
+\left(1-\frac{1}{d+1}\right) ^{3/2}
-2-\frac{3}{4(d+1)^{2}}
\right],
\end{equation*}
it becomes clear that (\ref{eq:q-upper}) is equivalent to
\begin{equation}
\frac{64(d+1)^{2}}{3}
\left[
\left(1+\frac{1}{d+1}\right) ^{3/2}
+\left(1-\frac{1}{d+1}\right) ^{3/2}
-2-\frac{3}{4(d+1)^{2}}
\right]
< \frac{1}{d+3/2}\log \frac{d+1}{d}.
\label{ineq}
\end{equation}
Now let $0 < t < 1$ such that $t = 1/(d+1)$, for every integer $d \geq 2$.
Inequality (\ref{ineq}) becomes
\begin{equation}
\frac{64}{3t^{4}}
\left[
(1+t)^{3/2}+(1-t)^{3/2}-2-\frac{3}{4}t^{2}
\right]
< \frac{-\log (1-t)}{t(1+t/2)}.
\label{eq:curvature-series-comparison}
\end{equation}
It therefore suffices to prove (\ref{eq:curvature-series-comparison}) for every $0 < t < 1$, which is stronger than the required statement.
The binomial expansion gives
\begin{equation*}
\frac{64}{3t^{4}}
\left[
(1+t)^{3/2}+(1-t)^{3/2}-2-\frac{3}{4}t^{2}
\right]
= \sum_{j = 0}^{\infty}\lambda_{j}t^{2j},
\qquad
\lambda_{j}
:= \frac{128}{3}\binom{3/2}{2j+4} > 0,
\end{equation*}
where
\begin{equation*}
\lambda_{0} = 1,
\qquad
\lambda_{1} = \frac{7}{24},
\qquad
\lambda_{2} = \frac{33}{256}.
\end{equation*}
For $j \geq 3$,
\begin{equation*}
\frac{\lambda_{j+1}}{\lambda_{j}}
= \frac{(4j+5)(4j+7)}{4(2j+5)(2j+6)}
< \frac{j+1}{j+2}.
\end{equation*}
Since $\lambda_{3} = 143/2048 < 1/12$, induction gives
\begin{equation}
\lambda_{j} < \frac{1}{3(j+1)},
\qquad
j \geq 3.
\label{eq:lambda-upper}
\end{equation}
Writing
\begin{equation*}
\int_{0}^{1}\frac{ds}{(1-st)(1+t/2)}
= \frac{-\log (1-t)}{t(1+t/2)}
= \sum_{m = 0}^{\infty}a_{m}t^{m}.
\end{equation*}
we obtain
\begin{equation*}
a_{m}
= \int_{0}^{1}\frac{s^{m+1}-(-1/2)^{m+1}}{s+1/2}\,ds.
\end{equation*}
Since $s+1/2 \leq 3/2$ on $[0,1]$,
\begin{equation*}
a_{2j}
> \frac{2}{3}\int_{0}^{1}s^{2j+1}\,ds
= \frac{1}{3(j+1)}.
\end{equation*}
Moreover,
\begin{equation*}
a_{2j+1}
= \sum_{k = 0}^{j}4^{-k}
\left[
\frac{1}{2j+2-2k}
-\frac{1}{2(2j+1-2k)}
\right]
\geq 0.
\end{equation*}
Finally,
\begin{equation*}
a_{0} = 1 = \lambda_{0},
\qquad
a_{2} = \frac{1}{3} > \frac{7}{24} = \lambda_{1},
\qquad
a_{4} = \frac{19}{120} > \frac{33}{256} = \lambda_{2}.
\end{equation*}
Together with (\ref{eq:lambda-upper}), these inequalities compare the two convergent series coefficient by coefficient and prove (\ref{eq:curvature-series-comparison}).
Substitution of $t = 1/(d+1)$ gives (\ref{eq:q-upper}).
\end{proof}

\begin{proof}[Proof of Theorem~\protect\ref{thm:main}]
Put
\begin{equation*}
C_{0}
:= \gamma +6\log 2-\frac{10}{3},
\qquad
E_{n}
:= 8\pi n^{3/2}(n+1)^{3/2}\Delta_{n}-\log n.
\end{equation*}
The desired lower bound is equivalent to $E_{n} > C_{0}$.
We first prove that $(E_{n})$ is strictly decreasing.
Set $d = n+1$.
Equations (\ref{eq:scaled-delta-increment}) and (\ref{eq:alpha-Y}), followed by the estimates of Lemmas~\ref{lem:alpha-upper} and~\ref{lem:q-upper}, give
\begin{align*}
8\pi q_{d-1}Y_{d}
& = 8\pi q_{d-1}d^{3/2}\alpha_{\mathbb{C}}(d) \\
& < \frac{64}{3}q_{d-1}\sqrt{d}(d+1/2) \\
& < \log \frac{d}{d-1}.
\end{align*}
The first inequality has the stated direction because $q_{d-1} > 0$.
Therefore
\begin{equation*}
E_{d}-E_{d-1}
= 8\pi q_{d-1}Y_{d}
-\log \left(1+\frac{1}{d-1}\right)
< 0.
\end{equation*}
It remains to identify the limit.
The leading term in (\ref{eq:delta-asymptotic}) gives
\begin{equation*}
E_{d-1}
= 8\pi (d-1)^{3/2}d^{3/2}\Delta_{d-1}
-\log (d-1)
= C_{0}+O\left(\frac{\log (d-1)}{d-1}\right).
\end{equation*}
Hence $E_{d-1} \to C_{0}$.
Since $(E_{d-1})$ decreases strictly to $C_{0}$, we have $E_{d-1} > C_{0}$ for every $d \geq 2$, proving the lower bound for $\Delta_d$ in (\ref{eq:main-bound}).
Finally,
\begin{equation*}
\ell_{d}-\frac{10}{3}
= \log d+C_{0}
> 0,
\end{equation*}
because $6\log 2 > 4$ and $\gamma > 0$.
This proves the positivity of the lower bound and establishes the two lower inequalities in (\ref{eq:main-bound}).

We next prove the upper bound.
The lower bound just established, together with $\alpha_{\mathbb C}(d) \to 8/(3\pi)$, gives
\begin{equation}
\alpha_{\mathbb C}(d) > \frac{8}{3\pi},
\qquad
Y_d
= d^{3/2}\alpha_{\mathbb C}(d)
> \frac{8}{3\pi}d^{3/2}.
\label{eq:Y-lower-for-delta-upper}
\end{equation}
For $d \geq 2$, the definition (\ref{qd}) gives
\begin{equation*}
q_{d-1}
= 2d^{3/2}
\sum_{j = 2}^{\infty}
\binom{3/2}{2j}\frac{1}{d^{2j}}.
\end{equation*}
Every coefficient in this series is positive, so retaining the term $j = 2$ gives
\begin{equation}
q_{d-1} > \frac{3}{64d^{5/2}}.
\label{eq:q-lower-for-delta-upper}
\end{equation}
Combining this estimate with (\ref{eq:Y-lower-for-delta-upper}) yields
\begin{equation}
8\pi q_{d-1}Y_d > \frac{1}{d}.
\label{eq:increment-lower-for-delta-upper}
\end{equation}
The exact increment identity (\ref{eq:scaled-delta-increment}) therefore gives, for every $d\geq2$,
\begin{equation*}
E_{d-1}-E_d
= \log\frac{d}{d-1}-8\pi q_{d-1}Y_d
< \log\frac{d}{d-1}-\frac{1}{d}.
\end{equation*}
Since $E_d \to C_0$, summing from $d = n+1$ to infinity gives
\begin{equation*}
E_n-C_0
= \sum_{d = n+1}^{\infty}(E_{d-1}-E_d)
< \sum_{d = n+1}^{\infty}
\left(\log\frac{d}{d-1}-\frac{1}{d}\right)
= H_n-\log n-\gamma.
\end{equation*}
By the definitions of $E_n$ and $C_0$, this is the upper inequality in (\ref{eq:main-bound}).

It remains to estimate the difference between the two bounds.
The definition of Euler's constant gives
\begin{equation*}
H_n-\log n-\gamma
= \sum_{k = n}^{\infty}
\left\{\log\left(1+\frac{1}{k}\right)-\frac{1}{k+1}\right\}.
\end{equation*}
For every $x > 0$,
\begin{equation*}
0 < \log(1+x)-\frac{x}{1+x} < \frac{x^2}{2(1+x)}.
\end{equation*}
Indeed, the derivative of the first difference is $x/(1+x)^2 > 0$.
The derivative of
\begin{equation*}
\frac{x^2}{2(1+x)}-\log(1+x)+\frac{x}{1+x}
\end{equation*}
is $x^2/(2(1+x)^2) > 0$, and both differences vanish at $x = 0$.
Taking $x = 1/k$ and summing gives
\begin{equation*}
0 < H_n-\log n-\gamma
< \frac{1}{2}\sum_{k = n}^{\infty}\frac{1}{k(k+1)}
= \frac{1}{2n}.
\end{equation*}
Substitution into the exact difference between the endpoints of (\ref{eq:main-bound}) proves (\ref{eq:delta-enclosure-width}) and completes the proof.
\end{proof}

\section{Proof of Corollary \protect\ref{CorAlpha}}

\label{sec:consequences}

To prove the lower bound in Corollary~\ref{CorAlpha}, and hence the lower bound in (\ref{eq:cut-improved-bounds}), use (\ref{eq:alpha-limit}) and (\ref{eq:main-bound}):
\begin{equation}
\alpha_{\mathbb{C}}(d)-\frac{8}{3\pi}
= \sum_{k = d}^{\infty}
\left(
\alpha_{\mathbb{C}}(k)-\alpha_{\mathbb{C}}(k+1)
\right)
> \frac{1}{8\pi}
\sum_{k = d}^{\infty}
\frac{\ell_{k}-10/3}{k^{3/2}(k+1)^{3/2}}.
\label{alphaSum}
\end{equation}
Let $C_{0} := \gamma +6\log 2-\frac{10}{3}$ and
\begin{equation*}
g(x)
:= \frac{\log x+C_{0}}{(x+1)^{3}},
\qquad
x \geq 1.
\end{equation*}
Since $C_{0} > 2/3$, $g$ is positive and strictly decreasing on $[1,\infty)$ and $k^{3/2}(k+1)^{3/2} < (k+1)^{3}$,
\begin{equation*}
\frac{\ell_{k}-10/3}{k^{3/2}(k+1)^{3/2}}
> g(k)
> \int_{k}^{k+1}g(x)\,dx.
\end{equation*}
Summing over $k \geq d$ and integrating by parts gives
\begin{align*}
\sum_{k = d}^{\infty}
\frac{\ell_{k}-10/3}{k^{3/2}(k+1)^{3/2}}
& > \int_{d}^{\infty}
\frac{\log x+C_{0}}{(x+1)^{3}}\,dx \\
& = \frac{\ell_{d}-10/3}{2(d+1)^{2}}
+\frac{1}{2}
\left\{
\log \left(1+\frac{1}{d}\right)
-\frac{1}{d+1}
\right\}.
\end{align*}
Now (\ref{alphaSum}) gives
\begin{equation*}
\alpha_{\mathbb{C}}(d)
> \frac{8}{3\pi}+\Lambda_{d},
\end{equation*}
which proves the lower bound.

For the upper bound, set
\begin{equation*}
a_k := H_k+6\log 2-\frac{10}{3}.
\end{equation*}
Since $\log 2 > 1/2$, we have $a_k > 0$.
The upper bound in Theorem~\ref{thm:main} and the limit (\ref{eq:alpha-limit}) give
\begin{align}
\alpha_{\mathbb{C}}(d)-\frac{8}{3\pi}
& = \sum_{k = d}^{\infty}\Delta_k
< \frac{1}{8\pi}\sum_{k = d}^{\infty}
\frac{a_k}{k^{3/2}(k+1)^{3/2}} \notag \\
& < \frac{1}{16\pi}\sum_{k = d}^{\infty}a_k
\left(\frac{1}{k^2}-\frac{1}{(k+1)^2}\right).
\label{eq:alpha-upper-from-delta}
\end{align}
Since $k\geq1$, all quantities in the coefficient comparison are positive.
After multiplying through by the positive denominators, the second inequality is equivalent to $2\sqrt{k(k+1)} < 2k+1$.
This follows from $(2k+1)^2-4k(k+1) = 1$.
For $N \geq d$, summation by parts gives
\begin{equation*}
\sum_{k = d}^{N}a_k
\left(
\frac{1}{k^2}-\frac{1}{(k+1)^2}
\right)
= \frac{a_d}{d^2}
+\sum_{k = d+1}^{N}\frac{a_k-a_{k-1}}{k^2}
-\frac{a_N}{(N+1)^2}.
\end{equation*}
Here $a_k-a_{k-1} = 1/k$ and $a_N/(N+1)^2\to0$.
Letting $N \to \infty$ in (\ref{eq:alpha-upper-from-delta}) therefore gives
\begin{equation}
\alpha_{\mathbb{C}}(d)-\frac{8}{3\pi}
< \frac{1}{16\pi}
\left\{
\frac{H_d+6\log2-10/3}{d^2}
+\sum_{k = d+1}^{\infty}\frac{1}{k^3}
\right\}.
\label{eq:alpha-harmonic-upper}
\end{equation}
The proof of Theorem~\ref{thm:main} established
\begin{equation}
H_d < \log d+\gamma+\frac{1}{2d}.
\label{eq:harmonic-elementary-upper}
\end{equation}
We also have
\begin{equation}
\sum_{k = d+1}^{\infty}\frac{1}{k^3}
< \frac{1}{2(d+1/2)^2}.
\label{eq:cubic-tail-elementary-upper}
\end{equation}
Indeed, strict convexity of $x \mapsto x^{-3}$ and Jensen's inequality on each interval $[k-1/2,k+1/2]$ give
\begin{equation*}
\frac{1}{k^3} < \int_{k-1/2}^{k+1/2}\frac{dx}{x^3}.
\end{equation*}
Summing over $k \geq d+1$ proves (\ref{eq:cubic-tail-elementary-upper}).
Combining (\ref{eq:alpha-harmonic-upper}), (\ref{eq:harmonic-elementary-upper}), and (\ref{eq:cubic-tail-elementary-upper}) yields
\begin{align*}
\alpha_{\mathbb{C}}(d)-\frac{8}{3\pi}
& < \frac{1}{16\pi}
\left\{
\frac{\ell_d-10/3}{d^2}
+\frac{1}{2d^3}
+\frac{1}{2(d+1/2)^2}
\right\} \\
& = \frac{\ell_d-17/6}{16\pi d^2}
+\frac{3d+1}{32\pi d^3(2d+1)^2}
= \Omega_d.
\end{align*}
This proves (\ref{CorAlphaBound}), completing the proof.

\section*{Acknowledgments}

P.P.\ thanks the students in his Spring 2026 graduate course on theoretical statistics and machine learning\footnote{\url{https://pratikpatil.io/teaching/sds391p6-s26}} for discussions on an open problem posed in Bandeira's lecture notes~\cite[Open Problem~1.2, pp.~16--17]{bandeira2016ten}.
L.D.A.\ was supported by FWF Project 10.55776/PAT8205923.

ChatGPT (GPT-5.6 Pro) assisted with numerical experiments, verification and simplification of proofs, exact symbolic computation of the leading asymptotic coefficients and inequalities (which have been explicitly computed up to $O\left(\frac{\log d}{d^{6}}\right)$ in experiments not documented in the manuscript, where all computations are done manually and explained in detail), literature searches, and proofreading.
The authors assume responsibility for all content.

\bibliographystyle{abbrvurl}
\bibliography{references_v33}

\end{document}